\documentclass[11pt]{article}
\usepackage[T1]{fontenc}
\usepackage[utf8]{inputenc}
\usepackage[greek,english]{babel}

\oddsidemargin \evensidemargin
\usepackage[style=alphabetic, maxnames=4]{biblatex}
\usepackage{microtype}
\usepackage{enumitem}
\setenumerate{label=\textup{(\roman*)}}

\usepackage{amsmath, mathtools}
\usepackage{amssymb}
\usepackage{amsthm, thmtools}
\usepackage{tikz-cd}
\usepackage{braket}
\usepackage{bm}
\usepackage{xspace}

\usetikzlibrary{positioning}
\usepackage{alphabeta}
\DeclarePairedDelimiter\abs{\lvert}{\rvert}
\usepackage[frozencache]{minted}
\setminted[julia]{frame=lines,
rulecolor=\color{white!80!black},
fontsize=\small,
numbers=right,
numbersep=-5pt,
obeytabs=true,
encoding = utf8,
tabsize=4}

\usepackage{todonotes}
\usepackage{csquotes}

\usepackage[colorlinks]{hyperref}
\usepackage[nameinlink]{cleveref}
\hypersetup{
  linkcolor=[rgb]{0.3,0.3,0.6},
  citecolor=[rgb]{0.2, 0.6, 0.2},
  urlcolor=[rgb]{0.6, 0.2, 0.2}
}

\declaretheorem[numberwithin=section]{problem}
\declaretheorem[numberwithin=section]{theorem}
\declaretheorem[numberlike=theorem]{lemma, corollary,proposition}
\declaretheorem[numberlike=theorem,style=definition]{definition}

\declaretheorem[numberlike=theorem,style=remark]{remark}
\declaretheoremstyle[
notefont=\bfseries, notebraces={(}{)}
]{conjecture}

\newcommand{\ie}{i.\,e.\ }

\newcommand{\numberset}{\mathbb}

\newcommand{\C}{\mathbb{C}}

\newcommand{\N}{\numberset{N}}

\newcommand{\R}{\numberset{R}}
\newcommand{\K}{\mathbb{K}}
\newcommand{\PP}{\mathbb{P}}

\definecolor{cof}{RGB}{219,144,71}
\definecolor{pur}{RGB}{186,146,162}
\definecolor{greeo}{RGB}{91,173,69}
\definecolor{greet}{RGB}{52,111,72}
\definecolor{darkspringgreen}{rgb}{0.09, 0.45, 0.27}

\newcommand{\discup}{\mathbin{\dot\cup}}

\DeclareMathOperator{\Gr}{Gr}

\DeclareMathOperator{\Ker}{Ker}
\DeclareMathOperator{\Diag}{Diag}
\DeclareMathOperator{\diag}{diag}
\DeclareMathOperator{\GL}{GL}

\DeclareMathOperator{\SL}{SL}
\DeclareMathOperator{\PGL}{PGL}
\DeclareMathOperator{\Sym}{Sym}
\DeclareMathOperator{\SO}{SO}
\DeclareMathOperator{\Skew}{Skew}

\newcommand{\Id}{I}
\newcommand{\transpose}{\textup{\textsf{T}}}
\DeclareMathOperator{\proj}{\mathbb{P}}

\title{Mukai lifting of self-dual points in \texorpdfstring{$\PP^6$}{P6}}
\author{Barbara Betti and Leonie Kayser}
\date{}

\newcommand\blankfootnote[1]{%
  \begin{NoHyper}
  \renewcommand\thefootnote{}\footnote{#1}%
  \addtocounter{footnote}{-1}%
  \end{NoHyper}
}

\begin{document}

\maketitle

\begin{abstract}
A set of $2n$ points in $\PP^{n-1}$ is self-dual if it is invariant under the Gale transform. Motivated by Mukai's work on canonical curves, Petrakiev showed that a general self-dual set of $14$ points in $\PP^6$ arises as the intersection of the Grassmannian $\Gr(2,6)$ in its Plücker embedding in $\PP^{14}$ with a linear space of dimension $6$. In this paper, we focus on the inverse problem of recovering such a linear space associated to a general self-dual set of points. We use numerical homotopy continuation to approach the problem and implement an algorithm in \texttt{Julia} to solve it. Along the way, we also implement the forward problem of slicing Grassmannians and use it to experimentally study the real solutions to this problem.
\end{abstract}

\section{Introduction}

\blankfootnote{\textit{Keywords.} Mukai Grassmannian, self-dual points, numerical algebraic geometry, homotopy continuation}\blankfootnote{\textit{2020 Mathematics Subject Classification.} 13H10, 14D22, 14N05, 65H14}In 1988 Mukai proved that the intersection of the Grassmannian $\Gr(2,6)$ in its Plücker embedding in $\PP^{14}$ with a general linear space of dimension $7$ is a general curve of genus $8$ in its canonical embedding \cite{MUKAI1988357, mukai1992curves}. This is part of his more general structure results concerning canonical curves and K3 surfaces: Mukai established that the general K3 surface of genus $g\leq 10$ is a linear section of what is now called the \emph{Mukai Grassmannians} $X_g$. In the case of genus $g=8$ the Mukai Grassmannian is the classical Grassmannian $X_8 = \Gr(2,6)$ in its Plücker embedding. Prime Fano $3$-folds, K3 surfaces, canonical curves and self-dual point configurations arise as the intersection of  $\Gr(2,6)$ with a linear subspace of $\PP^{14}$ of dimension $9,8,7$ and $6$, respectively. The \emph{Mukai lifting problem}, introduced in \cite[Section 4]{Geiger} as a "formidable challenge", concerns the inverse task, that is, it deals with recovering a linear section of the Grassmannian that produces a given variety. 

In 2009 Petrakiev showed that a general self-dual point configuration of $14$ points in $\PP^6$ arises as the intersection of $\Gr(2,6)$ with a linear space of dimension $6$ in $\PP^{14}$ \cite[Thm.\ 3.3, \Cref{thm:petrakiev} below]{Petrakiev}. 
Our paper presents an algorithm that solves the corresponding lifting problem.
\begin{problem} \label{Problem: Mukai lift}
Given such a set $\Gamma\subseteq \PP^6$, find a linear embedding $L\colon \PP^6 \overset{{\sim}}{\to} \mathbb{L}\subseteq\PP^{14}$ such that 
\begin{equation} \label{eq: Mukai_lift_P6}
      L(\Gamma) = \mathbb{L} \cap \Gr(2,6). 
\end{equation}
\end{problem}

The existence of a solution is guaranteed by the result of Petrakiev. However, the proof is not constructive. Obtaining a linear section is a challenging problem that involves hard polynomial systems which have not been successfully solved before.
Our implementation relies on numerical homotopy continuation to solve equations, as well as exploiting the toric degeneration of the Grassmannian $\Gr(2,6)$. A similar approach is followed in a more general setting in \cite{DBLP:journals/moc/BurrSW23}, but we avoid using the defining ideal of the Grassmannian. Instead we directly degenerate its coordinate ring by using a finite Sagbi basis. Such a basis is the analogue of Gröbner bases for subalgebras of polynomial rings. 

The name self-dual stems from the defining property of these configurations being invariant under the \emph{Gale transform}. This transform is an involution that takes a tuple $\Gamma$  of $m$ points in $\PP^{n-1}$ into a tuple $\Gamma'$ of $m$ points in $\proj^{m-(n-1)}$, defined up to projective linear transformations. The Gale transform was first studied in the general case of $m$ points in $\PP^{n-1}$ by Coble \cite{Coble}, while some specific cases of self-dual points have been studied earlier. A first example is \emph{Pascal's Theorem} stating that the vertices of two triangles circumscribed to the same conics lie on another conic. Hesse gave a similar result for $8$ points in $3$-space. The transform is named after David Gale, who used it in his studies on convex polyhedra, we refer to \cite{EISENBUD2000127} for a more detailed history of the Gale transform. Applications of Gale transforms are present in coding theory, where the Gale transform of a linear code gives the parity-check matrix, and in solving polynomial systems, where upper bounds for the number of solutions are given by studying the \emph{Gale dual system}, see \cite[Chapter 6]{sottile2011real}. Some more recent results involving Gale duality can be found in \cite{arzhantsev2017gale, Bihan2008, CAMINATA2018653}.

In this work we turn abstract and classical algebraic geometry into practice using state-of-the-art numerical tools. It is a first step towards solving more challenging lifting problems, for example lifting canonical curves. Our main contribution, other than the algorithmic implementation, regards new normal forms for self-dual points and the parameterization of their moduli space. We also experimentally study the real geometry of the slicing problem.

The paper is organized as follows.
In \Cref{A primer of self-dual points} we give equivalent characterizations for self-dual points and briefly introduce the moduli spaces of point configurations and the Gale transform. We discuss normal forms of self-dual points and how to parametrize them generically. We finally recall Mukai's results and collect known descriptions of the moduli space of self-dual points $\mathcal{A}_n$ for small $n$.
In \Cref{Slcing using HC} we introduce the homotopy continuation method and use it to compute the points of in given linear section of the Grassmannian $\Gr(2,6)$, exploiting its Gröbner degeneration. We also experimentally analyze how many real solutions there exist.
In \Cref{Lifting using HC} we finally describe a set of polynomial equations representing the Mukai lifting problem and solve them for a general configuration of points using Homotopy Continuation.
\Cref{Implementation} describes the implementation of our algorithm as well as an extended example and performance results. The full code is implemented in the programming language \texttt{Julia} \cite{bezanson2012julia} version $1.9.3$ and it is available at the MathRepo page
\[
\texttt{\url{https://mathrepo.mis.mpg.de/MukaiLiftP6}}
\]

\section{A primer of self-dual points} \label{A primer of self-dual points}

In this section we give an elementary introduction to the theory of self-dual points. Our main goal is to define the moduli space $\mathcal{A}_n$ of self-dual points in $\PP^n$ and give parametrizations thereof. Our main contribution is the skew normal form of a sufficiently general set of self-dual points, which, together with the orthogonal normal form, give two useful normal forms when working with self-dual points. While the treatment is non self-contained, all details relevant for our algorithmic considerations are given.

\subsection{Self-dual point configurations}

We consider an algebraically closed field $\K$ of any characteristic, in the later sections we will specialize to $\K = \C$. Let $S = \K[X_1,\dots,X_n]$ be the homogeneous coordinate ring of $\PP^{n-1} = \PP(\K^n)$ and let $I(X)$ be the homogeneous ideal of a given set $X \subseteq \PP^{n-1}$. Recall that a set of points $\Gamma \subseteq \PP^{n-1}$ is (linearly) non-degenerate if they are not contained in a hyperplane, namely, $I(\Gamma)_1 = 0$. A set of points is linearly general if any subset of at most $n$ points is linearly independent. Let $\Diag(n) \subseteq \GL(n)$ be the set of invertible diagonal matrices and let $\Gamma  \subseteq  \PP(\K^n)$ be a  non-degenerate ordered set of $2n$ distinct points, represented by the columns of a $n\times 2n$ matrix. 

\begin{lemma}\label{lem:selfdual_characterization}
The following are equivalent for $\Gamma$:
\begin{enumerate}
\item All subsets of $2n-1$ points impose the same number of conditions on quadrics as $\Gamma$, in symbols $I(\Gamma)_2 = I(\Gamma\setminus\{\gamma\})_2$ for all $\gamma \in \Gamma$;
\item There exists an invertible diagonal matrix $\Lambda\in \Diag(2n)$ such that $\Gamma\cdot \Lambda \cdot \Gamma^\transpose = \bm{0}$;
\item There exists a partition $\Gamma = \Gamma_1 \discup \Gamma_2$ and a non-degenerate symmetric bilinear form $Q \in \Sym(n)$ such that both $\Gamma_i$ are orthogonal bases with respect to $Q$, \ie $\Gamma_i^\transpose \cdot Q \cdot \Gamma_i \in \Diag(n)$.
\end{enumerate}
If $\Gamma$ is linearly general, these are equivalent to
\begin{enumerate}[resume]
\item[\textup{(iii')}] For any partition $\Gamma = \Gamma_1 \discup \Gamma_2$ into $n+n$ points there is a non-degenerate bilinear form $Q \in \Sym(n)$ such that both $\Gamma_i$ are orthogonal bases with respect to $Q$.
\end{enumerate}
\end{lemma}

\begin{proof}
\begin{itemize}[wide]
\item[(ii)$\Rightarrow$(i)]
Let $A = \set{\alpha \in \N^n | \abs{\alpha} = 2}$, for a point $x \in \PP(\K^n)$ define $x^A \coloneqq (x^\alpha)_{\alpha \in A}$ as the row vector of monomials $X^\alpha$ evaluated in $z$. Then the linear map $E \colon\K[X_1,\dots,X_n]_2 \to \K^{2n}$ given by evaluation of quadrics in $\Gamma$ is represented by
\[
E = \begin{bmatrix}
    (\gamma_i)^\alpha
\end{bmatrix}_{1\leq i\leq 2n,\abs \alpha = 2} = \begin{bmatrix}
- (\gamma_1)^A - \\
\vdots \\
- (\gamma_{2n})^A -
\end{bmatrix}, \qquad \Gamma = [\gamma_1|\dots|\gamma_{2n}].
\]
Its kernel $\Ker E \subseteq \K[X_1,\dots,X_n]_2$ consists of quadrics passing through $\Gamma$. We need to show that the rank of $E$ does not change when deleting any of the $2n$ rows from $E$. For a point $\gamma$, the vector $\gamma^A$ consists of the upper half of the matrix $\gamma\cdot \gamma^\transpose$, and by (ii) we have $\sum_i \Lambda_{ii} \gamma_i\gamma_i^\transpose = \bm 0$. By assumption all $\Lambda_{ii}$ are nonzero, so any row of $E$ is a linear combination of the other rows, proving the rank condition.
\item[(i)$\Rightarrow$(ii)] Reversing the argument, by (i) every $\gamma_j\gamma_j^\transpose$ is a linear combination
\[
\gamma_j\gamma_j^\transpose = \sum_{i\neq j} \beta_i^{(j)}\gamma_i\gamma_i^\transpose.
\]
Thus the system of equations $\Gamma \Lambda \Gamma^\transpose = \bm 0$, linear in diagonal matrices $\Lambda$, has solutions $\diag(\beta^{(j)})$ with $\beta_j^{(j)} = -1 \neq 0$. Since $\K$ is infinite, there exists a solution with all entries nonzero.
\item[(ii)$\Rightarrow$(iii):] After rescaling the columns by $\gamma_i \mapsto 1/\sqrt{\Lambda_{ii}} \gamma_i$ we may assume that $\Lambda = \Id_{2n}$. After reordering the columns we may assume that $\Gamma = [\Gamma_1|\Gamma_2]$ where $\Gamma_1$ is invertible. Consider the non-degenerate symmetric matrix $Q \coloneqq \Gamma_1^{-\transpose}\Gamma_1^{-1}$, then using
\[
\bm 0 = \Gamma \Gamma^\transpose = \Gamma_1\Gamma_1^\transpose + \Gamma_2\Gamma_2^\transpose
\]
we see that $\Gamma_2$ is also invertible and $Q = -\Gamma_2^{-\transpose}\Gamma_2^{-1}$, and hence $\Gamma_i^\transpose Q \Gamma_i = \pm\Id_n$.
\item[(iii)$\Rightarrow$(ii):]This argument is reversible after rescaling $\Gamma$ to match $\Gamma_i^\transpose Q \Gamma_i = \pm\Id_n$.
\item[(ii)$\Leftrightarrow$(iii'):] The same argument clearly applies to any partition if $\Gamma$ is linearly general. \qedhere 
\end{itemize}
\end{proof}

\begin{definition}
A set of $2n$ points in $\PP^{n-1}$ is \emph{self-dual} or self-associated if it satisfies any of the equivalent conditions (i),(ii),(iii) from the previous lemma.
\end{definition}
The name \enquote{self-dual} is related to the \emph{Gale transform} on the space of point-configurations, which we will introduce in the next section. Its fixed points are exactly self-dual point configurations. A deeper characterization of self-dual points related to the Gorenstein property of $S/I(\Gamma)$ is given by Eisenbud and Popescu \cite{EISENBUD2000127}. A finite-dimensional $\K$-algebra $A$ is Gorenstein if it is injective as a module over itself. An equivalent characterization is the existence of a $\K$-linear map $e\colon A \to \K$ such that $(f,g) \mapsto e(fg)$ is a perfect pairing.
A standard graded $\K$-algebra $S_X$ of Krull dimension $n$ is Gorenstein if there is a regular sequence $\ell_1,\dots,\ell_n \in (S_X)_1$ such that $S_X/\langle \ell_1,\dots,\ell_n\rangle$ is Gorenstein.

\begin{theorem}[{\cite[Theorem 7.1 \& 7.3]{EISENBUD2000127}}] \label{thm:gorenstein}
For a non-degenerate set of $2n$ points $\Gamma \subseteq \PP^{n-1}$ the following are equivalent:
\begin{enumerate}
\item $\Gamma$ is self-dual and fails by $1$ to impose independent conditions on quadrics;
\item every subset of $2n-1$ points imposes independent conditions on quadrics, but $\Gamma$ does not impose independent conditions;
\item $\Gamma$ is arithmetically Gorenstein, \ie the graded ring $S_\Gamma = S/I(\Gamma)$ is Gorenstein.
\end{enumerate}
\end{theorem}

\subsection{The space of point configurations}

Following \cite[Chapter 3]{Mumford1996GIT} or \cite[Chapter II]{Dolgachev1988}, the moduli space $\mathcal{P}^m_n$ of ordered sets of $m$ points in $\PP^n$ can be constructed as a GIT quotient as follows:
\begin{itemize}
\item The product $(\PP^n)^{m}$ admits group actions of $G=\SL(n+1)$ and  $\mathfrak{S}_m$ via
\[
g \cdot (x_1,\dots,x_m) = (gx_1,\dots,gx_m), \qquad \sigma \cdot (x_1,\dots,x_m) = (x_{\sigma^{-1}(1)},\dots,x_{\sigma^{-1}(m)}).
\]
\item Consider the very ample line bundle on $(\PP^n)^{m}$
\[
\mathcal{L} \coloneqq \mathcal{O}_{(\PP^n)^{m}}(1,\dots,1) = \bigotimes_{i=1}^m \pi_i^* \mathcal{O}_{\PP^n}(1).
\]
This enables us to embed $(\PP^n)^{m}$ as a projective variety in $\PP^{(n+1)^m-1}$ with homogeneous coordinate ring $\bigoplus_{k\geq 0} H^0((\PP^n)^{m}, \mathcal{L}^{\otimes k})$. The bundle $\mathcal{L}$ linearizes with respect to the action of $\SL(n+1)$ in the sense of \cite{Mumford1996GIT}.
\item With respect to this linearized action, a tuple $Z \in (\PP^n)^{m}$ is semi-stable if and only if
\[
\dim \operatorname{Span}_{\PP^n}(Y)  +1 \geq \frac{(n+1)\cdot \#Y}{m} \qquad \text{for all }Y \subseteq Z.
\]
The tuple $Z$ is stable if and only if this inequality is always strict \cite[Chapter II Thm.\ 1]{Dolgachev1988}.
\item Consider the graded ring of invariants $R^m_n \coloneqq \bigoplus_{k\geq 0} H^0((\PP^n)^{m}, \mathcal{L}^{\otimes k})^G$, then
\[
\mathcal{P}^m_n \coloneqq\operatorname{Proj} R^m_n  =  ((\PP^n)^{m})^{\rm ss} \mathbin{/\!/\!_{\mathcal{L}}} G.
\]
\end{itemize}
The moduli space of ordered point configurations $\mathcal{P}^m_n$ is a natural domain for the Gale transformation, which we now introduce, a reference is \cite[Chapter III]{Dolgachev1988}.

\begin{definition}
The \emph{Gale transform} of $\Gamma\in \mathcal{P}_n^m$ is given by any ordered set $G(\Gamma)\in \mathcal{P}^m_{m-n-2}$ of $m$ points in $\PP^{m-n-2}$, represented by a matrix of size $(m-n-1)\times m $ such that there exists an invertible diagonal matrix $\Lambda\in{\rm Diag}(m)$ verifying
\begin{equation} \label{Gale_transform_def}
 \Gamma\cdot \Lambda\cdot G(\Gamma)^\transpose=0. \end{equation}
In other words, the \emph{Gale transform} is the map between moduli spaces of ordered set of points
\[
G\colon \mathcal{P}_n^m \longrightarrow \mathcal{P}^m_{m-n-2}, \qquad \Gamma\longmapsto G(\Gamma).
\]
\end{definition}
The Gale transform of $\Gamma$ is given by the transpose of a kernel matrix of $\Gamma$. Defining the moduli space $\mathcal{P}_n^m$ as above makes this construction well-defined for any semi-stable configuration. Assuming that $\Gamma = [\Id_{n+1}| A]$ and taking the diagonal matrix $\Lambda$ in the definition as $\Id_{m} \oplus-\Id_{m}$, it follows that the Gale transform of $\Gamma$ is given by the configuration $G(\Gamma)\in\mathcal{P}^m_{m-n-2} $ represented by the columns of the matrix  $[A^\transpose| \Id_{m-(n+1)}]$.  

\begin{proposition}
$\Gamma\in \mathcal{P}^{2n}_{n-1}$ is self-dual if and only if it is fixed under the Gale transform.
\end{proposition}
\begin{proof}
Inspecting \Cref{Gale_transform_def}, we see that $\Gamma = G(\Gamma)$ in $\mathcal{P}^{2n}_{n-1}$ is a reformulation to condition (ii) of \Cref{lem:selfdual_characterization}.
\end{proof}

The characterization using the Gale transform finally leads to the definition of the moduli space of self-dual points.

\begin{definition}
The \emph{moduli space of un-ordered self-dual point configurations} in $\PP^n$ is the fixed locus of the Gale transform $G$ on $\mathcal{P}^{2n+2}_n/\mathfrak{S}_{2n+2}$. It will be denoted by $\mathcal{A}_n$.
\end{definition}

\subsection{Normal forms for self-dual points} \label{sec:normal_form}

We consider the problem of parametrizing the space of self-dual points by means of suitable normal forms. \Cref{lem:selfdual_characterization}(iii) suggests a normal form using orthogonal matrices over $\K$
\[
\SO(n) \coloneqq \set{P \in \SL(n) | P \cdot P^\transpose = \Id_n }.
\]

\begin{theorem}[Orthogonal normal form] \label{orthogonal_normal_form}
Let $\Gamma = [\gamma_1,\dots,\gamma_{2n}] \subseteq \PP(\K^n)$ be an ordered set of $2n$ points with $\gamma_1,\dots,\gamma_n$ linearly independent. Then the following hold:
\begin{enumerate}
\item $\Gamma$ is self-dual if and only if there exists a matrix $A \in \PGL(n)$ such that, after rescaling the columns of $\Gamma$, we have $A\cdot\Gamma = [\Id_n | P]$ for some $P \in \SO(n)$. 
\item Moreover, the orthogonal matrix $P$ is unique up to multiplying the columns with $\varepsilon \in \{\pm1\}^n$ such that $\prod_i \varepsilon_i = 1$.
\end{enumerate}
\end{theorem}
\begin{proof}
\begin{enumerate}[wide]
\item If $A\cdot\Gamma = [\Id_n | P]$, then $\Gamma$ is self-dual by \Cref{lem:selfdual_characterization}(iii), for the converse we provide a construction of $A$ and $P$. By \Cref{lem:selfdual_characterization}(ii)  there exists an invertible diagonal matrix $\Lambda={\rm diag}(\lambda_1,\dots,\lambda_{2n})$ such that $\Gamma\cdot \Lambda \cdot \Gamma^\transpose = 0$. Rescaling the columns, \ie multiplying $\Gamma$ on the right with $\diag(\sqrt{\lambda_1},\dots,\sqrt{\lambda_n},\sqrt{-\lambda_{n+1}},\cdots,\sqrt{-\lambda_{2n}})$, we can assume that $\Lambda = \Id_n \oplus -\Id_n$. Since the first $n$ points of $\Gamma=[\Gamma_1|\Gamma_2]$ are linearly independent, the matrix $A:=\Gamma_1$ is invertible and $A^{-1}\cdot\Gamma=[\Id_n|P]$ is self-dual with respect the same diagonal matrix as $\Gamma$. Then
\[ \begin{bmatrix}
        \Id_n & P
    \end{bmatrix}\cdot \begin{bmatrix} \Id_n & \\ & -\Id_n   \end{bmatrix}  \cdot \begin{bmatrix}
        \Id_n \\ P^\transpose
    \end{bmatrix}= 
        \Id_n-P\cdot P^\transpose = 0,
\]
which shows that $P$ is an orthogonal matrix. If $P$ has determinant $-1$, we can rescale the last column of $\Gamma$ and $P$ to obtain a special orthogonal matrix.
\item The orthogonal matrix $P$ only depends on the choice of the square roots in the rescaling with $\operatorname{diag}(\sqrt{\lambda_1},\dots,\sqrt{-\lambda_{2n}})$, and hence is unique up to signs.
\qedhere
\end{enumerate}
\end{proof}

\begin{corollary}\label{cor:SOparamDeg}
Assume that $\operatorname{char}(\K) \neq 2$, we have a surjective morphism $\SO(n) \twoheadrightarrow \mathcal{A}_{n-1}$ generically finite of degree $2^{2n-2}(2n)!$.
\end{corollary}
A previous version of this paper contains a mistake in the degree computation, we thank Alessio Caminata for pointing this out to us.
\begin{proof}
The orthogonal normal form $P \mapsto \Gamma =  [I_n|P]$ is surjective by the previous theorem and the characterization of semi-stable points. This normal form of an ordered set of self-dual points is unique up to the action of $2^{n-1}$ ways of multiplying columns of $P\in \SO(n)$ by $\pm 1$ (action of $(\K^\times)^{2n}$ on columns of $\Gamma$) and the $2^{n-1}$ ways of multiplying the rows of $P$ by $\pm 1$ (action of $\SL(n)$ on rows); here we used $-1 \neq 1$. Moreover, for a general $\Gamma$, all its permutations by $\sigma \in \mathfrak{S}_{2n}$ have distinct normal forms, hence the rational map $\SO(n) \to \mathcal{A}_{n-1}$ is generically finite of degree $2^{2n-2}\cdot (2n)!$.
\end{proof}

From now on assume that $\operatorname{char}(\K) \neq 2$, there is no harm in restricting to $\K= \C$.
A parametrization of $\mathcal{A}_{n-1}$ more suitable for computation relies on the parametrization of the orthogonal group via the Cayley transform, originally introduced in \cite{Cayley1846}. Let $U \subseteq \K^{n\times n}$ be the set of matrices $A$ with $I_n+A$ invertible, then the \emph{Cayley transform} is the self-inverse isomorphism of varieties
\[
\mathcal{C} \colon U \to U, \qquad \mathcal{C}(A) = (I_n - A)(I_n + A)^{-1}.
\]
Noting that the product $(I_n - A)(I_n + A)^{-1}$ commutes, a straightforward computation shows that $\mathcal{C}$ induces a birational map between skew-symmetric and special orthogonal matrices:
\[
\mathcal{C} \colon \Skew(n)\cap U \leftrightarrow \SO(n) \cap U.
\]
Hence the orthogonal normal form $ [I_n \mid P]$ of a general set of points can be rewritten as $[I_n\mid \mathcal{C}(S)]$ for $S \in \Skew(n)$.

\begin{remark} \label{remark: eigenvalue 1}
Over any field, an orthogonal matrix of odd size \emph{always} has an eigenvalue equal to $1$ or $-1$ (if its determinant is, respectively, $1$ or $-1$). If the matrix is generic in $\SO(n)$, then we have that $-1$ is not an eigenvalue, thus the Cayley transform is well-defined for generic $P\in \SO(n)$.
\end{remark}
Since point configurations in $\mathcal{A}_{n-1}$ are only defined up to projective linear transformations, we can apply the transformation $I_n+S$ to obtain
\begin{equation} \label{Skew normal form}
[I_n\mid \mathcal{C}(S)] = [I_n\mid (I_n - S)(I_n + S)^{-1}] \quad \leadsto \quad [I_n + S \mid I_n - S].
\end{equation}
This avoids the delicate step of inverting a matrix in the implementation of this parametrization, reducing computation time and numerical error. This step was suggested to us by Simon Telen and we refer to the representation in \Cref{Skew normal form} as a \emph{skew normal form} (SNF).
    
\begin{theorem}[Skew normal form] \label{Skew_normal_form}
A general configuration of self-dual points $\Gamma\in \mathcal{A}_{n-1}$ admits a skew normal form. Moreover, the dominant rational map
\[
\Skew(n) \dashrightarrow \mathcal{A}_{n-1}, \qquad S \mapsto [I_n+S \mid I_n-S]
\]
is generically finite of degree $(2n)!2^{2n-2}$.
\end{theorem}
Since a skew-symmetric matrix $S\in \Skew(n)$ is uniquely determined by the $\frac{n(n-1)}{2}$ entries in the upper triangular part, a suitable parameter space for $\mathcal{A}_{n-1}$ is $\C^{\frac{n(n-1)}{2}}$:
\[
(s_1,\dots,s_{\frac{n(n-1)}{2}}) \; \mapsto \; \left[
\begin{array}{cccc|cccc}
     1 &s_1&\cdots & s_{n-1}  & 1 &-s_1&\cdots & -s_{n-1} \\
     -s_1 & 1 & \ddots & \vdots & s_1 & 1 & \ddots & \vdots \\
     \vdots & \ddots&  \ddots & s_{\frac{n(n-1)}{2}} & \vdots & \ddots &  \ddots & -s_{\frac{n(n-1)}{2}} \\
     -s_{n-1} & \cdots & -s_{\frac{n(n-1)}{2}} & 1 & s_{n-1} & \cdots & s_{\frac{n(n-1)}{2}} & 1
\end{array}
\right].
\]
In our situation, $14$ general self-dual points in $\PP^6$ have exactly
\[
2^{12} \cdot 14! = 357.082.280.755.200
\]
distinct orthogonal or skew normal forms.

\subsection{Mukai-type results for \texorpdfstring{$\mathcal{A}_n$}{An}}\label{sec:mukai-results}

For small values of $n$, self-dual configurations in $\PP^n$ are well-studied. A related landmark result is Mukai's description of polarized K3 surfaces and (canonical) curves of genus $g$ as complete intersections on homogeneous varieties $X_g$ \cite{MUKAI1988357, mukai1992curves}. He proved that a general polarized K3 surface $(S,\mathcal{L})$ of genus $g$, meaning $\mathcal{L}$ is ample and $\mathcal{L}^2 = 2g-2$, arises as the intersection of homogeneous varieties $X_g \subseteq \PP^N$ with general hypersurfaces (in most cases hyperplanes) if $g\leq 10$. Similarly, general canonical curves $C \subseteq \PP^{g-1}$, embedded by their very ample canonical bundle $\omega_C$, arise as sections of $X_g$ if $g\leq 9$.

In dimension zero, general self-dual point configurations in $\PP^{g-2}$ arise as sections of $X_g$ if $g \leq 8$. This is classical for $g\leq 6$ (see \cite[Part II]{EISENBUD2000127} for a modern account) and was proven by Petrakiev for $g=7,8$ \cite{Petrakiev}. In particular, all self-dual point configurations in $\PP^n$ for $n\leq 6$ arise as linear sections of canonical curves by Mukai's result. For $g\geq 9$ this is no longer true by a simple moduli count. A summary of these results is given in the following theorem.

\begin{theorem} \label{thm:mukai_grass}
\begin{enumerate}
\item All sets of four points in $\PP^1$ are self-dual.
\item Six points in $\PP^2$ are self-dual if and only if they are the intersection of a quadric and cubic.
\item A general set in $\mathcal{A}_3$ is a complete intersection of three quadric surfaces in $\PP^3$.
\item A general set in $\mathcal{A}_4$ is a quadric section of an elliptic normal curve in $\PP^5$; equivalently it is the intersection of $\Gr(2,5) \subseteq \PP^9$ with a quadric and a linear space.
\item A general set in $\mathcal{A}_5$ is a linear section of the Lagrangean Grassmannian $X_7 = \operatorname{LG}_+(5,10) \subseteq \PP^{15}$.
\item A general set in $\mathcal{A}_6$ is a linear section of the Grassmannian $X_8 = \Gr(2,6) \subseteq \PP^{14}$.
\end{enumerate}\medskip
In cases \textup{(ii)--(vi)}, an equivalent statement is that the set $\Gamma \subseteq \PP^{n-1}$ is a hyperplane section of a canonical curve of genus $n+1$ in $\PP^n$.
\end{theorem}

Part (iv)--(vi) all correspond to a non-trivial lifting problem, the last one being exactly \Cref{Problem: Mukai lift}. For reference, we record a precise statement of Petrakiev's result on $\mathcal{A}_6$. The group $\SL(6)$ acts on $\C^{15} = \bigwedge^2\C^6$ and hence induces an action on $\Gr(7,\C^{15})$, leaving $X_8 \subseteq \PP(\C^{15})$ invariant.

\begin{theorem}[{\cite[Thm.\ 3.3]{Petrakiev}}] \label{thm:petrakiev}
The slicing map
\[
f \colon \Gr(7,\C^{15})/\SL(6) \dashrightarrow \mathcal{A}_6, \qquad \mathbb{L} \mapsto \mathbb{L} \cap X_8
\]
defines a generically finite dominant rational map to the moduli space of self-dual points $\mathcal{A}_6$.
\end{theorem}

\section{Slicing using homotopy continuation} \label{Slcing using HC}

\subsection{Introduction to homotopy continuation}

Homotopy continuation is a computational method used to solve a multivariate polynomial system $F(x)=0$ with a zero-dimensional set of solutions over the complex numbers. It relies on finding a homotopy that deforms $F(x)$ into a system that is easier to solve.
Usually the system is parameterized introducing some variables $\underline{a}$ giving a family of systems $F(x,\underline{a})$ with the same number of solutions for every $\underline{a}$. The system  we want to solve corresponds to the target parameter $F(x,\underline{a}_{\text{target}}) = F(x)$ and the one that we are able to solve corresponds to the start parameter $F(x,\underline{a}_{\text{start}})$. Defining a homotopy between these two systems, we move the start parameter and, numerically tracking paths, we deform the solution of $F(x,\underline{a}_{\text{start}})$ into the solution of $F(x,\underline{a}_{\text{target}})$. 
An example of homotopy between two systems $G$ and $F$ is given by a straight line between them:
\[
H(x,u)= (1-u)G(x) + uF(x) \qquad u\in[0,1].
\]
This technique has been implemented by Breiding and Timme \cite{breiding2018homotopycontinuationjl} in the \texttt{Julia} package \texttt{HomotopyContinuation.jl}, as the following snippet of code demonstrates.
\begin{minted}{julia}
using HomotopyContinuation as HC              
@var x[1:n] a[1:m]
f = [f_1,...,f_l] # Define equations f_i(x,a) appropriately
F = System(f, variables = x[1:n], parameters = a[1:m])       
res = HC.solve(F, start_sol; start_parameters = a_start, target_parameters = a_target)
\end{minted}

For a polynomial system there are infinitely many possible start systems and homotopies. Different choices may have a different impact on the computational time of the \texttt{solve()} function. If the start system has the same number of solutions as $F(x)$, we say that it is \emph{optimal}.
Optimal start systems are preferred because they require to track the minimal number of paths possible and still get all solutions of $F(x)=0$. This choice also shortens significantly the computation time. However, in general there is no way to predict which start system and homotopy will lead to the quickest time to solve. It is also necessary to ensure that the parameterized system $F(x,\underline{a})$ has the same number of solutions. In other words, we need a parameterization $F(x,\underline{a})$ such that we can keep track of all the start solutions along the way. This is a fundamental challenge in the Mukai lifting problem, because it requires to parameterize the space of self-dual points in $\PP^6$ (and not just arbitrary point configurations). 

\subsection{Computing a linear section of the Grassmannian} \label{Equations for the linear section}

In this section we compute a set of self-dual points $\Gamma$ as  a given $6$-dimensional linear subspace $\mathbb{L}$ intersecting the Grassmannian $X_8\coloneqq \Gr(2,6)\subseteq \PP^{14}$. By construction, the linear space $\mathbb{L} \subseteq \PP^{14}$ is a solution to the Mukai lifting problem for $\Gamma$. In the next section we set up a homotopy in which $\Gamma$ serves to get the start parameter $\underline{a}_{\text{start}}$ and $\mathbb{L}$ serves to get the start solution. 
We recall that $X_8$ parameterizes $2$-dimensional subspaces in $\C^6$. We represent an element in an affine chart of $X_8$ as the row span of the $2\times 6$ matrix:
\begin{equation} \label{matrix_grass}
     H = \begin{bmatrix}
     1 & 0 &  t_1 & t_2 & t_3 & t_4 \\
     0 & 1 & t_5 & t_6 & t_7 & t_8
    \end{bmatrix}.  
\end{equation}
The Plücker embedding $p\colon \Gr(2,6) \hookrightarrow \PP^{14}$, $p(H)= (p_0:\dots:p_{14}) \in \PP^{14} $ is given by taking the $2\times 2$ minors of $H$, these are the $15$ Plücker coordinates in $\C[t_1,\dots,t_8]$. In our affine chart we have $p_0=1$.
\begin{lemma} \label{Lemma_8 line eq}
      Let $A \in \C^{8\times 15}$ be a general matrix and $\mathbb{L} = \PP(\Ker A)$. The polynomial system in $\underline{t}$
\begin{equation} \label{lin_eq}
A  \cdot \big(p_0(\underline{t}),\dots, p_{14}(\underline{t}) \big)^\intercal = 0
\end{equation}
has exactly $14$ solutions $\underline{t}_1,\dots,\underline{t}_{14}$ and the set $Z = \set{p(\underline{t}_k) | k=1,\dots,14 }$ is the image of a configuration of self-dual points $\Gamma\subseteq \proj^6$ under a linear map $L\colon\PP^6 \overset{{\sim}}{\to} \mathbb{L}\subseteq \PP^{14}$.
\end{lemma}

\begin{proof}
The Grassmannian $X_8$ is a $8$-dimensional variety of degree $14$ \cite[Thm~5.13]{Michalek2021Invit-53971}, so its intersection with $8$ linear hyperplanes in general position contains exactly $14$ points.
Let  $L$ be a kernel matrix for $A$. If $\underline{t}_{k}$ is a solution to \eqref{lin_eq}, then $p(\underline{t}_k)$ is a linear combination of the columns of $L$ and it lies in the image $\mathbb{L}\coloneqq {\rm Im}(L)$ of the linear map $L\colon \PP^6\rightarrow \proj^{14}$. Since $X_8\subseteq \PP^{14}$ is arithmetically Gorenstein, so is $Z\subseteq \mathbb{L}$ and by \Cref{thm:gorenstein}(iii) there exists a self dual set of points $\Gamma\subseteq \proj^6$ such that $ L(\Gamma)=P$. To compute such a configuration it is enough to compute a solution $\Gamma$ to:
\[
L \cdot \Gamma = \begin{bmatrix}        p_0(\underline{t}_1) &  & p_0(\underline{t}_{14}) \\
                      \vdots & \cdots & \vdots \\
                      p_{14}(\underline{t}_1) &  & p_{14}(\underline{t}_{14})
\end{bmatrix}. \qedhere
\]
\end{proof} 
We point out that the previous Lemma gives an algorithm to construct the self-dual configuration $\Gamma$ in a linear section $\mathbb{L}$ defined by $8$ hyperplanes $\sum_{j=1}^{15} a_{ij}x_j=0$, $i=1,\dots,8$. 
After fixing a random matrix $A$, we use a toric degeneration of the Grassmannian to compute solutions $\underline{t}_k$ for \eqref{lin_eq}. We dedicate the following section to describe the construction of the degeneration.

\subsection{Toric degeneration of the Grassmannian}

The homogeneous coordinate ring $R \coloneqq S_X\subseteq \K[t_1,\dots,t_n]$ of a projective variety $X \subseteq \PP^{n-1}(\K)$ can be studied by looking at the \emph{initial algebra} of $R$ with respect to a term order $\prec$, that is, the algebra generated by the initial terms with respect to $\prec$ of every element in $R$
\[
{\rm in}_\prec(R) \coloneqq \K[{\rm in}_\prec(f):f\in R].
\]
Many geometric properties of the variety $X$ can be understood by looking at the corresponding algebraic properties of the initial algebra ${\rm in}_\prec(R)$. For example, if the initial algebra is normal, Cohen-Macaulay or Gorenstein, then so is the coordinate ring $R$. We refer to \cite{bruns2022determinants} for further details. Although the coordinate ring of a projective variety is a finitely generated algebra, the corresponding initial algebra may not be finitely generated for a given term order. This happens when $R$ does not admit a finite Khovanskii basis for that term order. The existence of universal finite Sagbi basis is an active research problem, see \cite{BRUNS2024102237} for some recent developments.

\begin{definition}
    Let $R\subseteq \K[t_1,\dots,t_n]$ be a finitely generated polynomial algebra. A set of generators $\mathcal{F}$ of $R$ is a \emph{Khovanskii basis} with respect to a term order $\prec$ if ${\rm in}_\prec(R)= \K[{\rm in}_\prec(f):f\in \mathcal{F}]$.
\end{definition}
Khovanskii bases have been introduced by Robbiano and Sweedler with the name of \emph{Sagbi bases} in \cite{robbiano2006subalgebra} and they are called \emph{Canonical Subalgebra bases} by Sturmfels in \cite{sturmfelsgröbner}. Later, Kaveh and Manon in \cite{doi:10.1137/17M1160148} gave a more general definition introducing the terminology that we use, even if the context of this paper is in the Sagbi case. Khovanskii bases are the analog of Gröbner bases for ideals to subalgebras and they have been effectively used in computer algebra in \cite{BETTI2025102340}. When the coordinate ring $R$ admits a finite Khovanskii basis, the initial algebra is a finitely generated monomial algebra, hence it can be viewed as the coordinate ring of a toric variety. In this case we say that $X$ admits a \emph{toric degeneration} to the projective toric variety ${\rm Proj}({\rm in}_\prec(R))$.

\begin{definition}
A toric degeneration of $X\subseteq \PP^n$ is a variety $Y\subseteq \PP^n \times \mathbb{A}^1$ such that the projection $\pi\colon Y \to \mathbb{A}^1$ is flat, $\pi^{-1}(0) \hookrightarrow \PP^n$ is a toric variety and $\pi^{-1}(\lambda) \cong X$ for every $\lambda\neq 0$.
\end{definition}

Degenerations are used to formalize the concept of limit of a family of varieties as the \emph{special fiber} $\pi^{-1}(0)$ can be thought of as the limit of $\pi^{-1}(t)$ for $t\rightarrow 0$.
Now we illustrate explicitly this degeneration for the Grassmannian $X_8$. For a \emph{weight vector} $\omega$, we define the $\omega$-initial form ${\rm in}_\omega(f)$ of a polynomial $f=\sum c_\alpha t^\alpha$ to be the sum  of the terms $c_\alpha t^\alpha$ such that $\alpha\cdot\omega$ is maximal. A weight vector $\omega$ represents a term order $\prec$ for a finite set of polynomials $\mathcal{F}$ if ${\rm in}_\omega(f)={\rm in}_\prec(f)$ for every $f\in\mathcal{F}$. 
The Plücker coordinates generate the coordinate ring $R$ of $X_8 \cap \{p_0= 1\} \subseteq \mathbb{A}^{14}$, that is $R=\C[p_0(\underline{t}),\dots,p_{14}(\underline{t})]$. A \emph{diagonal term order} is a term order for which the initial terms of the Plücker coordinates $p_i$ are the corresponding monomial in a main diagonal of \eqref{matrix_grass}.
\begin{proposition}[{\cite[Prop.~11.8]{sturmfelsgröbner}}]
    The Plücker coordinates are a finite Khovanskii basis for the coordinate ring $R$ of the Grassmannian $X_8$ with respect to any diagonal term order $\prec_{\mathrm d}$, which is represented by the weight vector $\omega = (3,2,1,0,0,1,2,3)$ on the Plücker coordinates: 
\[
{\rm in}_\omega(R)= \K[ {\rm in}_\omega(p_0),\dots, {\rm in}_\omega(p_{14})].
\]
\end{proposition}
We introduce a new variable $u$, that will be the parameter of the degeneration. For every polynomial $f= \sum_{\alpha} c_\alpha t^\alpha \in R$, we define $\nu(f)\coloneqq \max\set{\omega\cdot\alpha | c_\alpha\neq 0}$ and
\begin{equation} \label{f_u}
    \begin{aligned}
    f_u &\coloneqq \sum\limits_{\alpha} c_\alpha t^\alpha\cdot u^{\nu(f)-\omega\cdot\alpha } \in \K[t_1,\dots,t_8,u], \\
     R_u &\coloneqq \K[f_u: f\in R] \subseteq \K[t_1,\dots,t_8,u].
    \end{aligned}
\end{equation}
We observe that $f_0= {\rm in}_\omega(f)= \lim\limits_{u\rightarrow 0} f_u$ and $f_1=f$. In particular we recover the coordinate ring and its initial algebra as $R_0={\rm in}_\omega(R)$ and $R_1=R$. By moving the parameter $u$ from $1$ to $0$ we degenerate the coordinate ring $R$ into a monomial algebra. The latter is the coordinate ring of a toric variety and this corresponds to degenerating $X_8$ into the toric variety $\operatorname{Spec}(R_0)$.
\begin{remark}
    The previous construction is an application of the more general result by Anderson \cite{cite-key} regarding toric degenerations of projective varieties whose coordinate ring has a full-rank valuation with finitely generated value semigroup. Taking $\omega$ as before, the full-rank valuation
 \[
 \nu \colon R\setminus{0}\rightarrow \mathbb{Z}^8,  \qquad \nu \Big(\sum\limits_{\alpha} c_\alpha t^\alpha \Big) = \operatorname{argmax} \Set{\omega\cdot \alpha | c_\alpha\neq 0}.
 \]
 induces a term order represented by the vector $\omega$. The existence of a finite Khovanskii basis for this term order is equivalent to having a finitely generated value semigroup for $\nu$  by \cite{doi:10.1137/17M1160148}.
\end{remark}

Going back to \Cref{lin_eq}, we introduce the parameter $u$ to degenerate it to a toric ideal as follows. We denote by $\underline{a}$ the entries of the matrix $A$ from \Cref{Lemma_8 line eq} and define the system $F(\underline{t},u, \underline{a})$ with $8$ equations by
\begin{equation} \label{F_u}
F(\underline{t},u, \underline{a})=\{ A\cdot (p_{0,u}(\underline{t}),\dots, p_{14,u}(\underline{t}))^\intercal=0 \}.
\end{equation} 
The variables $u, \underline{a}$ are the parameters in the homotopy, the equations in \eqref{lin_eq} correspond to the target system $F(\underline{t},1, \underline{a})$. Fixing random values $\underline{a}_{\text s}$ we obtain the start system $F(\underline{t},0, \underline{a}_{\text s})$ that involves only linear relations among monomials, i.e. the leading terms of the Plücker coordinates. 
\texttt{HomotopyContinuation.jl} can easily solve this by automatically choosing a polyhedral homotopy. Then we track the solutions to obtain the solutions of $F(\underline{t},1, \underline{a}_{\text s})$. Given a matrix $A$ with entries $\underline{a}_{\text t}$, we solve $F(\underline{t},1,\underline{a}_{\text t})$ as the following snippet of \texttt{Julia} code shows.

The equations in \eqref{F_u} are encoded in the vector \texttt{F}. The first homotopy computes the solution on the toric degeneration by moving the hyperplanes that intersect it and the parameter $u=0$ is fixed. The second homotopy transports the solution from the toric degeneration to the Grassmannian $X_8$ by increasing the parameter $u$ from $0$ to $1$ and fixing the parameters $\underline{a}$. The third homotopy finally solves $F(\underline{t},1,\underline{a}_{\text t})$ by tracking the solution of $F(\underline{t},1,\underline{a}_{\text s})$.
\begin{minted}{julia}
C         = System(F, variables = x, parameters = [u;a[:]])

A_rand    = randn(ComplexF64, length(a));
a0_start  = [0;A_rand]
toric_sol = solutions(HC.solve(C,            # solutions of F(t,0,a_s)
                target_parameters = a0_start)) 

a1_start  = [1;A_rand]
start_sol = solutions(HC.solve(C, toric_sol; # solutions of F(t,1,a_s)
                start_parameters = a0_start, target_parameters = a1_start))
            
a1_target = [1;A] # given A
X8_cap_A  = solutions(HC.solve(C, start_sol; # solutions of F(t,1,a_t)
                start_parameters = a1_start, target_parameters = a1_target)) 
\end{minted} 
Once we have the solutions $P$ to $F(\underline{t},1, \underline{a}_t)$, we obtain a configuration of self-dual points $\Gamma$ that is embedded in $\proj^{14}$ by solving the linear system $L( \Gamma )=P$, where $L \coloneqq \Ker([\underline{a}_t])$ is a kernel basis. This is illustrated in the following snippet of code, the function \texttt{numerical\char`_plücker} (numerically) evaluates the Plücker embedding on points in the affine chart $H \cong \C^8$.
\begin{minted}{julia}
Z = reshape(vcat(numerical_plücker.(X8_cap_A)...),15,14) # Embed points in P^14
L = LinearAlgebra.nullspace(A) # Obtain matrix L whose image is Ker(A)
Γ = L\Z # Take the inverse image of Z under L
\end{minted}

\subsection{Real solutions}

For a transversal slicing of the Grassmannian $X_8$ with $8$ complex hyperplanes, as described in Equation \eqref{lin_eq}, we get $14$ complex solutions. We study how many of these are real for random choices of real hyperplanes. The regions in $\Gr_{\R}(7,15)$ corresponding to transversal linear spaces are partitioned into open semi-algebraic cells by the number of real intersection points.

We compute the solutions with \texttt{HomotopyContinuation.jl} using $F(\underline{t},1,\underline{a}_t)$ as start system, for which we computed a start solution \texttt{X8\_cap\_A} previously. The hyperplanes are chosen as follows: we draw the coefficients of the defining equations from a standard normal distribution with mean 0 and standard deviation 1.  This gives the uniform distribution on the Grassmannian $\Gr_{\R}(7,15)$ with respect to the Haar measure from the compact real Lie group $\operatorname{O}(15)$ \cite[Thm.\ 2.2.1]{Chikuse2003-tz}. In \Cref{table: sample} we report the result of the sampling after ten million iterations. We observe that the case in which all $14$ solutions are real is very rare but it does happen. The columns \textit{succ} and \textit{fail} report the number of times the computations succeeded in computing all the $14$ distinct solutions and the number of times it failed.

\begin{table}[ht]
    \centering
    \scalebox{0.88}{
    \begin{tabular}{ccccccccccc}
        Real solutions &0 & 2 & 4 & 6 & 8 & 10 & 12 & 14 & succ & fail\\ \hline
        Count    & 178244  & 2277545 & 4079864 & 2547002 & 790369 & 116267 & 10508 &192 & 9999991 & 9  \\
        Proportion & 1.78\%  & 22.78\% & 40.80\% &  25.47\%& 7.90\%& 1.16\% & 0.11\% & 0.002\% & 100\% & -
    \end{tabular}
    }
    \caption{Slicing $X_8$ with uniformly sampled real $\mathbb{L} \subseteq \PP^{14}(\R)$ $10^7$ times.}\label{table: sample}
    \label{tab:my_label}
\end{table}

\section{Lifting using homotopy continuation}\label{Lifting using HC}

In the previous section we computed pairs $(\mathbb{L},Z)$ with $X_8 \cap \mathbb{L} = Z \subseteq \PP^{14}$. In this section these will be turned into the start solution of a parameterized polynomial system solving the lifting problem \Cref{Problem: Mukai lift}. As mentioned, an important step for using Homotopy Continuation is parameterizing the space of self-dual points $\mathcal{A}_6$ in $\mathbb{P}^6$. The main problem is that this space is not convex (in an obvious way) and, more generally, it is crucial to find a path between two self-dual configurations that stays \emph{inside} the space of self-dual points. Using the parametrization given by $\C^{21}\dashrightarrow \mathcal{A}_6$ that exploits the skew normal form in \Cref{Skew_normal_form}, we solve this problem efficiently: after computing normal forms for start and target configuration, we can choose a straight line through the parameter space $\C^{21}$. If $s_{\text{start}}$ is chosen sufficiently randomly, this path will avoid the discriminant (branch locus) as it has real codimension $2$.

With this structure, the Mukai lifting problem becomes the following: For $S\in \Skew(7)\cong \C^{21}$ we want to recover the embedding $L\colon \PP^6\overset{\sim}{\to} \mathbb{L}\subseteq \PP^{14}$ such that   
\[
L \cdot \{ \gamma \in [ I_7 + S \mid I_7 - S] \} = \mathbb{L} \cap X_8.
\]

\begin{lemma}     \label{thm: final sys}
Let $\gamma_i$ be the $i$-th column of a set $\Gamma$ in skew normal form and let $q_1,\dots q_{15}$ be the Plücker relations of $X_8$, \ie the generators of the defining ideal of $X_8\subseteq \PP^{14}$. The embedding $L=[\ell_{ij}]\colon \PP^6 \overset{\sim}{\to} \mathbb{L}\subseteq\PP^{14}$ is a solution for the Mukai lifting problem for $\Gamma$ if and only if
     \begin{equation} \label{final_sys}
          q_k\left( L \gamma_i \right) = 0 \qquad \text{for all }i=1,\dots,14,\; k=1,\dots,15. 
     \end{equation}
\end{lemma}
\begin{proof}
The only condition that $L$ has to satisfy is that the image of $\Gamma$ lies inside the Grassmannian $X_8$. This is ensured by the vanishing of the Plücker relations on $L\cdot \gamma_i$ for every $i=1,\dots,14$. Indeed if \Cref{final_sys} holds, then $L(\Gamma)\subseteq \mathbb{L}\cap X_8$ and we have equality as the latter is a set of $14$ points by \Cref{Lemma_8 line eq} and $L$ is injective.
\end{proof}
We observe that the system in \eqref{final_sys} has degree $2$ in the $105$ variables $\ell_{ij}$ and has $14\cdot 15 = 210$ equations, so it is heavily over-determined. Additionally, this system is not zero-dimensional, which is a necessary condition for applying homotopy continuation methods. Indeed, Petrakiev's result \cite{Petrakiev} implies that the rational map
\[
f \colon \Gr(7,\C^{15}) \dashrightarrow \mathcal{A}_6, \qquad \mathbb{L} \mapsto \mathbb{L} \cap X_8
\]
generically has $35$-dim'l fibers (unions of orbits of the $\SL(6)$ action on $\C^{15} = \bigwedge^2 \C^6$).

Let $\widehat{\mathcal A}_6 \subseteq (\PP^6)^{14}$ be the locally closed set of semistable self-dual point configuration. In \Cref{Slcing using HC} we considered the related rational map (note that we considered kernel matrices instead):
\[
\hat{f} \colon \C^{7\times 15} \dashrightarrow \widehat{\mathcal{A}}_6/\mathfrak{S}_{14}, \qquad L \mapsto L^{-1}({\rm Im}(L)\cap X_8) \subseteq \PP^6.
\]
The relationship between $f$, $\hat{f}$ and the skew normal forms can be summarized as follows.
\begin{theorem} \label{thm:diagram}
We have a commutative diagram
\[
\begin{tikzcd}[column sep=0.5ex]
 &                  & \widehat{\mathcal{A}}_6 \arrow[d, two heads] \arrow[r, "\subseteq", phantom]                  & (\PP^6)^{14}\phantom{/\mathfrak{S}_{14}}           &                                                                                                          \\
\C^{7\times 15} \arrow[rr, "\hat f", two heads, dashed] \arrow[d, two heads, dashed]                        &                  & \widehat{\mathcal{A}}_6/\mathfrak{S}_{14} \arrow[dd, two heads] \arrow[r, "\subseteq", phantom] & (\PP^6)^{14} / \mathfrak{S}_{14}                   & \Skew(7) \arrow[llu, "\textup{SNF}"', two heads, dashed] \arrow[lldd, "\text{gen.fin}", two heads, dashed] \\
{\Gr(7,\C^{15})} \arrow[d, two heads, dashed]                                                    & \circlearrowleft &                                                                                               &                                                    &                                                                                                          \\
{\Gr(7,\C^{15})/\SL(6)} \arrow[rr, "f", two heads, dashed] \arrow[rr, "\text{gen.fin}"', dashed] &                  & \mathcal{A}_6 \arrow[r, "\subseteq", phantom]                                                 & \mathcal{P}_6^{14}/\mathfrak{S}_{14}\phantom{()^6} &                                                                                                         
\end{tikzcd}
\]
In particular, $\hat{f}$ is dominant with general fiber of dimension $\dim \hat{f}^{-1}(\Gamma) = 36$.
\end{theorem}
\begin{proof}
Commutativity of the diagram follows from the previous discussion. Since $f$ is dominant (\Cref{thm:petrakiev}) and the image of $\hat f$ is invariant under $\SL(7)$, $\hat f$ is dominant too. The dimension of a general fiber is
\[
\dim \hat{f}^{-1}(\Gamma) = \dim \C^{7\times 15} - \dim \widehat{\mathcal{A}}_6 = 105 - 69 = 36.\qedhere
\]
\end{proof}

To get a system with a finite set of solutions we need to impose a number of conditions on the entries of $L$ equal to the dimension of a general fiber of $\hat f$. Hence we can construct $L$ as follows using $69$ variables $\ell_i$, $i=1,\dots,69$. As described in \Cref{Slcing using HC} we consider a self-dual start configuration $\Gamma_{\text{start}}$ embedded in the linear space $\mathbb{L}_{\text{start}}\coloneqq L_{\text{start}}(\PP^6)\subseteq \PP^{14}$. By adding random linear combinations of the $69$ variables to each entry of $L_{\text{start}}$, given by random matrices $A_k \in \C^{7\times 15}$ for $k=1,\dots, 69$, we construct $L$ as
\[
L = L_{\text{start}} + \sum\limits_{k=1}^{69}\ell_k \cdot A_k.
\]
With this interpretation, the start solution for \texttt{HomotopyContinuation.jl} is a vector of zeros $\ell_{\text{start}} = (0,\dots,0)$ and the start parameters are the entries of the skew-symmetric matrix $S_{\text{start}}$ in a skew normal form of $\Gamma_{\text{start}}$.

The number of solutions to this system could be larger than the degree of the slicing map $f$ (which is conjecturally birational), because we are intersecting $\SL(6)$-orbits with linear spaces. Since we only want to construct \emph{some} Mukai lift, we do not investigate this subtlety further, though it would be interesting to study this degree in future work. 

We finally describe part of the construction of \Cref{final_sys} in \texttt{Julia}. We concatenate three different systems as follows:
\begin{itemize}
\item First we define \texttt{L\_system} to be the system containing the equations of the image of a linear map $L\colon \PP^6\rightarrow \PP^{14}$, as in \Cref{thm: final sys} and $L$ as described above in the variables $\texttt{l} \mathrel{\hat =} (\ell_1,\dots,\ell_{69})$. The variables \texttt{x} are the coordinates of a point in $\PP^6$.
\item Then we evaluate the Plücker relations in these expressions, obtaining the system \texttt{Q\char`_system} of $15$ equations in \texttt{l} and \texttt{x}.
\item Finally, we evaluate the equations of \texttt{Q\char`_system} in each of the points in $\Gamma_{\texttt{s}} = [I_7 + S | I_7 - S]$, keeping the variables $\texttt{s} \mathrel{\hat =} S \in \Skew(7)$ as parameters.
\end{itemize}
A relevant part of the construction is summarized in the following snippet:

\begin{minted}{julia}
L_system = System(L*x, variables=[x;l])
plücker_system = System([oscar_to_HC_Q(plück_oscar[i], q) for i=1:15], variables=q)
Q = plücker_system(expressions(L_system));
Q_system = System(Q, variables=[x;l]);
equations = vcat([Q_system([Γ[:,i];l]) for i in 1:14]...);
parameterized_system = System(equations, variables=l, parameters=s);
\end{minted}
The full implementation together with the example notebook presented in the next section can be found at
\[
\text{\url{https://mathrepo.mis.mpg.de/MukaiLiftP6}}
\]

\subsection{Towards testing Petrakiev's Conjecture} \label{sec:testing_petrakiev}

Petrakiev's dominant rational map $f \colon \Gr(7,\C^{15})/\SL(6) \dashrightarrow \mathcal{A}_6$ from \Cref{thm:petrakiev} is generically finite. In the same paper the author conjectures that $f$ is actually birational \cite[Remark 2.11]{Petrakiev}. In this section we indicate how one can use our implementation to test this conjecture.

Given a general point configuration $\Gamma \in \mathcal{A}_6$, the question is whether for any two $\mathbb{L}_1,\mathbb{L}_2 \in \Gr(\PP^6,\PP^{14})$ with $\mathbb{L}_i \cap \Gr(2,6) = \Gamma$ (in the sense that there exist linear isomorphisms $\PP^6 \overset{\sim}{\to} \mathbb{L}_i$ sending $\Gamma$ to $\mathbb{L}_i \cap \Gr(2,6)$), there is a matrix $G \in \SL(6)$ with $G\cdot\mathbb{L}_1 = \mathbb{L}_2$. Here, the action of $\SL(6)$ is the one induced by $\PP^{14} = \PP(\bigwedge^2\C^6)$ described in \Cref{sec:mukai-results}. We can separate this test into the following steps:
\begin{enumerate}[label=(\arabic*)] \setcounter{enumi}{-1}
\item Pick a general $\Gamma \in \mathcal{A}_6$.
\item Compute two preimages $\mathbb{L}_1,\mathbb{L}_2$ of $\Gamma$ under $\Gr(\PP^6,\PP^{14}) \dashrightarrow \mathcal{A}_6$.
\item Try to compute a matrix $G \in \SL(6)$ with $G \cdot \mathbb{L}_1 = \mathbb{L}_2$.
\end{enumerate}
Step (0) can be achieved using any of the normal forms discussed in \Cref{sec:normal_form}, most conveniently a skew normal form obtained from a random complex skew matrix.

To get one preimage one can run the lifting algorithm described in the previous section. To obtain a second preimage $\mathbb{L}_2$, one can run a monodromy-based approach on this parametrized polynomial system \cite{MartndelCampo2017}. Since the graph of $\Gr(7,\C^{15})/\SL(6) \dashrightarrow \mathcal{A}_6$ is irreducible, one can obtain any other element in the fiber by tracking the first element along a homotopy with start and end point $\Gamma$. To obtain such paths, one can take any path connecting two skew normal forms of $\Gamma$, for example by piecing together line segments in $\Skew(7)$. The following describes one such path: From an initial starting parameter $S_0$, move to $S_1$ and then in a triangle $S_1\to S_2 \to S_3 \to S_1$.

\begin{center}
\begin{tikzpicture}

\node[circle, fill=black, label=below left:{$\textcolor{orange}{S_0}$}, scale=0.5] (0) at (-2,0) {};

\node[circle, fill=black, label=below:{$\textcolor{orange}{T_0}=\textcolor{blue}{T_3}=\textcolor{red}{S_1}$}, scale=0.5] (1) at (0,0)     {};

\node[circle, fill=black, label=above:{$\textcolor{blue}{S_3}=\textcolor{darkspringgreen}{T_2}$}, scale=0.5] (3) at (1,1.5) {};

\node[circle, fill=black, label=below right:{$\textcolor{red}{T_1}=\textcolor{darkspringgreen}{S_2}$}, scale=0.5] (2)  at (2,0) {};

\draw[thick, orange, ->] (0.east) -- (1.west);
\draw[thick, red, ->] (1.east) -- (2.west);
\draw[thick, darkspringgreen, ->] (2.north west) -- (3.south east);
\draw[thick, blue, ->] (3.south west) -- (1.north east);

\end{tikzpicture}
\end{center}
Recall that our algorithm doesn't directly work on $\Gr(7,\C^{15})/\SL(5)$, but rather in an affine subspace of $\C^{7\times 15}$ which maps generically finitely to $\widehat{\mathcal{A}}_6/\mathfrak{S}_{14}$ under $\hat{f}$ (\Cref{thm:diagram}). This enables us to compute suitable $\mathbb{L}_2$, therefore making step (1) possible. An example of two such matrices are also included in the aforementioned MathRepo page.

Step (2), on the other hand, is an instance of an orbit membership problem: Representing $\mathbb{L}_i$ as matrices $L_i \in \C^{15\times 7} \cong (\bigwedge^2\C^6) \otimes \C^7$, one needs to decide the existence
\[
\exists G \in \SL(6), B \in \GL(7) \qquad\text{such that} \qquad B \cdot L_1 \cdot {\wedge^2 G} = L_2 \quad ?
\]
This is a difficult algorithmic problem, and has been studied intensively in the context of orbit closure problems in computational invariant theory \cite{Derksen2002}. We leave the computational solution of this numerical algebraic geometry problem for future work.

\section{Implementation and extended example}\label{Implementation}

  In this section we explain the implementation through an instructive example. We consider the self-dual configuration $\Gamma \subseteq \PP^6$ given by the columns of the following matrix:

\[
\Gamma = \begin{bsmallmatrix}
7 & -2 & 6 & -1 & -6 & 1 & -9 & 7 & 0 & 6 & 1 & 8 & -3 & 7 \\
-1 & 2 & -5 & -2 & 0 & -4 & 3 & -3 & -4 & -3 & 4 & -2 & 4 & -1 \\
1 & 4 & -1 & -5 & -3 & 6 & 8 & -1 & -8 & -3 & 5 & 1 & -6 & -8 \\
3 & -6 & 4 & -3 & -4 & 6 & 0 & 5 & 8 & 2 & 3 & 2 & -8 & 0 \\
1 & -2 & 1 & 0 & -4 & 2 & 2 & 3 & 4 & -1 & 2 & 2 & -2 & -2 \\
0 & -6 & -5 & 6 & 3 & 7 & -3 & 2 & 8 & -7 & -6 & -3 & -5 & 5 \\
-3 & 3 & -4 & 1 & 4 & 3 & 2 & -3 & -6 & -4 & -3 & -4 & -1 & -2
\end{bsmallmatrix}
\]
To verify that $\Gamma$ is indeed self-dual, we use condition (ii) of \Cref{lem:selfdual_characterization} and compute an invertible diagonal matrix $\Lambda$ verifying $\Gamma\cdot\Lambda\cdot\Gamma^\intercal = 0$. This is a linear system in $14$ unknowns (the entries of $\Lambda$) and $\binom{7+1}{2} = 28$ equations, namely $\sum_{i=1}^{14} \lambda_i \gamma_i\gamma_i^\intercal = 0$. The function \texttt{certify\_selfdual} returns a floating point approximation of the (exact) normed solution
\[
\Lambda = \frac{1}{\sqrt{13.0625}} \diag(-1,-1,-1,-1,-1,-1,-1,1,\textstyle\frac{1}{4},1,1,1,1,1).
\]
\begin{minted}{julia}
Lambda = certify_selfdual(Gamma)
norm( Gamma * Lambda * transpose(Gamma), Inf ) #2.6867397195928788e-14
\end{minted}
We bring $\Gamma$ into an orthogonal normal form $\Gamma_{\text{ONF}}$ as in \Cref{orthogonal_normal_form}, essentially using linear algebra. Applying an appropriate scaling $\Lambda_{\text{scale}}$ and linear map $A$ we have $A\cdot \Gamma_{\text{ONF}} = \Gamma \cdot \Lambda_{\text{scale}}$.
\begin{gather*}
\Gamma_{\text{ONF}} = [I_7\mid P]  = \begin{bsmallmatrix}
1 & 0 & 0 & 0 & 0 & 0 & 0 & -0.3399 & 0.06 & -0.0924 & 0.235 & 0.7564 & -0.3033 & 0.3911 \\
0 & 1 & 0 & 0 & 0 & 0 & 0 & -0.2052 & -0.8506 & 0.0197 & -0.1977 & 0.268 & 0.1942 & -0.2922 \\
0 & 0 & 1 & 0 & 0 & 0 & 0 & 0.1008 & 0.2302 & 0.8118 & -0.2422 & 0.3148 & -0.1311 & -0.3208 \\
0 & 0 & 0 & 1 & 0 & 0 & 0 & 0.0948 & 0.319 & -0.4084 & -0.722 & 0.3291 & 0.3032 & -0.0306 \\
0 & 0 & 0 & 0 & 1 & 0 & 0 & -0.7734 & 0.0615 & 0.1552 & -0.4091 & -0.3747 & -0.1609 & 0.2009 \\
0 & 0 & 0 & 0 & 0 & 1 & 0 & 0.2627 & -0.1669 & -0.2427 & -0.2875 & -0.072 & -0.8558 & -0.1552 \\
0 & 0 & 0 & 0 & 0 & 0 & 1 & -0.3951 & 0.2942 & -0.2868 & 0.2752 & 0.055 & -0.0546 & -0.7703
\end{bsmallmatrix} \\
A = \mathit{i} \begin{bsmallmatrix}
3.6821 & -1.052 & 3.1561 & -0.526 & -3.1561 & 0.526 & -4.7341 \\
-0.526 & 1.052 & -2.63 & -1.052 & 0.0 & -2.104 & 1.578 \\
0.526 & 2.104 & -0.526 & -2.63 & -1.578 & 3.1561 & 4.2081 \\
1.578 & -3.1561 & 2.104 & -1.578 & -2.104 & 3.1561 & 0.0 \\
0.526 & -1.052 & 0.526 & 0.0 & -2.104 & 1.052 & 1.052 \\
0.0 & -3.1561 & -2.63 & 3.1561 & 1.578 & 3.6821 & -1.578 \\
-1.578 & 1.578 & -2.104 & 0.526 & 2.104 & 1.578 & 1.052
\end{bsmallmatrix}    \qquad \Lambda_{\text{scale}} = \sqrt{\Lambda[1{:}7]} \oplus -\sqrt{\Lambda[8{:}14]}
\end{gather*}
\begin{minted}{julia}
Gamma_ONF, A, Lambda_scale = orthogonal_normal_form(Gamma)
norm( A * Gamma_ONF - Gamma * Lambda_scale, Inf ) # 2.6645352591003757e-15
\end{minted} 
We compute a skew normal form $\Gamma_{\text{SNF}}$ by applying the Cayley transform to $P$ to obtain the skew-symmetric matrix $S$ and we compute $\Gamma_{\text{SNF}} = [I_7+S | I_7-S]$. The implementation ensures that $P\in \SO(7)$ so that the Cayley transform is defined, as pointed out in \Cref{remark: eigenvalue 1}. The output is a floating point approximation of the following exact matrix:
\[
S = \begin{bsmallmatrix}
0 & -1 & 1 & 0 & -1 & 4 & 2 \\
1 & 0 & 3 & -2 & -2 & 10 & 12 \\
-1 & -3 & 0 & 2 & 1 & -1 & -2 \\
0 & 2 & -2 & 0 & -1 & -10 & -6 \\
1 & 2 & -1 & 1 & 0 & -4 & -4 \\
-4 & -10 & 1 & 10 & 4 & 0 & -6 \\
-2 & -12 & 2 & 6 & 4 & 6 & 0
\end{bsmallmatrix} \qquad \Gamma_{\text{SNF}} = \begin{bsmallmatrix}
1 & -1 & 1 & 0 & -1 & 4 & 2 & 1 & 1 & -1 & 0 & 1 & -4 & -2 \\
1 & 1 & 3 & -2 & -2 & 10 & 12 & -1 & 1 & -3 & 2 & 2 & -10 & -12 \\
-1 & -3 & 1 & 2 & 1 & -1 & -2 & 1 & 3 & 1 & -2 & -1 & 1 & 2 \\
0 & 2 & -2 & 1 & -1 & -10 & -6 & 0 & -2 & 2 & 1 & 1 & 10 & 6 \\
1 & 2 & -1 & 1 & 1 & -4 & -4 & -1 & -2 & 1 & -1 & 1 & 4 & 4 \\
-4 & -10 & 1 & 10 & 4 & 1 & -6 & 4 & 10 & -1 & -10 & -4 & 1 & 6 \\
-2 & -12 & 2 & 6 & 4 & 6 & 1 & 2 & 12 & -2 & -6 & -4 & -6 & 1
\end{bsmallmatrix}
\]
\begin{minted}{julia}
P = Matrix{Float64}(Gamma_ONF[:,8:14])
det(P) #1.0000000000000009
S = cayley(P)
Gamma_SNF = [I+S I-S]
\end{minted}
We compute the linear embedding $\widehat L\colon \PP^6 \overset{{\sim}}{\to} \mathbb{L} \subseteq \PP^{14}$ such that $\widehat L(\Gamma_{\text{SNF}}) = \mathbb{L} \cap X_8$. The parametrized polynomial system in \Cref{final_sys} has been precomputed as \texttt{F} as described in \Cref{Lifting using HC} together with start solution $\texttt{l\char`_start}=0$ and start parameters \texttt{S\char`_start}. The target parameter of the homotopy are the entries of the matrix $S$. We executed this
computation on a single thread of a 512 GB RAM machine using 2x 12-Core  Intel Xeon E5-2680 v3 processor working at 2.50 GHz in $3172$ seconds.
\begin{minted}{julia}
S_target = skew_to_vector(S) # unrolls the upper right half of S into its 21 entries
@time result = HomotopyContinuation.solve(F, l_start; 
              start_parameters=S_start, target_parameters=S_target)  #3172.87531 seconds
sol = solutions(result)[1]
L_hat = L_start + sum(sol[i]*A_rand[i] for i in eachindex(l))
\end{minted}

\[ \widehat L=
\begin{bsmallmatrix}
0.11+1.48\mathit{i} & -0.31-0.31\mathit{i} & -2.09-0.9\mathit{i} & 2.23-1.62\mathit{i} & -0.66-4.3\mathit{i} & 1.75-0.77\mathit{i} & 0.33-0.75\mathit{i} \\
2.37-1.0\mathit{i} & -0.66+0.3\mathit{i} & -0.77+1.73\mathit{i} & -0.59-1.0\mathit{i} & 0.5-0.2\mathit{i} & 0.52-1.16\mathit{i} & -0.02+0.15\mathit{i} \\
-0.37-1.9\mathit{i} & -0.34+0.46\mathit{i} & 1.07+0.61\mathit{i} & -0.27-0.56\mathit{i} & -0.91-1.21\mathit{i} & -0.37-0.29\mathit{i} & -0.17-0.63\mathit{i} \\
-1.74-1.36\mathit{i} & 0.57-0.67\mathit{i} & 0.71+0.34\mathit{i} & 0.77+1.58\mathit{i} & 3.09+0.01\mathit{i} & -0.6+1.76\mathit{i} & 1.07-0.91\mathit{i} \\
-0.86-0.15\mathit{i} & 0.42+0.27\mathit{i} & -0.09-0.18\mathit{i} & -0.19-0.36\mathit{i} & -2.18+0.17\mathit{i} & -0.65-0.3\mathit{i} & -0.09-0.15\mathit{i} \\
-0.26+2.25\mathit{i} & 1.12-0.89\mathit{i} & 0.92-1.52\mathit{i} & 4.07+1.13\mathit{i} & 0.87-0.98\mathit{i} & -1.68+1.66\mathit{i} & 1.67-0.93\mathit{i} \\
1.34+0.68\mathit{i} & -0.12-0.38\mathit{i} & -0.91-0.23\mathit{i} & -0.12-1.23\mathit{i} & 0.95-1.19\mathit{i} & 1.04+1.04\mathit{i} & -0.42-0.64\mathit{i} \\
0.19-1.7\mathit{i} & -1.69+0.8\mathit{i} & -0.39-0.27\mathit{i} & -2.99-0.63\mathit{i} & 0.67+2.89\mathit{i} & 0.45+0.98\mathit{i} & -1.3+1.11\mathit{i} \\
-0.43-0.07\mathit{i} & -0.93-0.57\mathit{i} & -1.16-1.07\mathit{i} & -0.12-1.08\mathit{i} & 1.06-2.44\mathit{i} & 2.42+0.22\mathit{i} & -0.73-0.21\mathit{i} \\
-0.34-0.7\mathit{i} & -0.09+0.37\mathit{i} & 0.49+0.02\mathit{i} & -1.14+0.16\mathit{i} & -0.27-0.16\mathit{i} & 0.47-0.24\mathit{i} & -0.65+0.31\mathit{i} \\
0.79+1.62\mathit{i} & -0.44-1.02\mathit{i} & 0.04-0.1\mathit{i} & 2.02-2.51\mathit{i} & -0.53-0.65\mathit{i} & 0.58+2.37\mathit{i} & -0.52-1.89\mathit{i} \\
0.54+0.89\mathit{i} & -0.2+0.19\mathit{i} & -1.1+0.96\mathit{i} & -0.45+0.8\mathit{i} & 0.53+0.7\mathit{i} & 0.35-1.09\mathit{i} & 0.18+0.8\mathit{i} \\
1.57-0.58\mathit{i} & -0.55+0.14\mathit{i} & -0.14+0.07\mathit{i} & -1.58-0.16\mathit{i} & 0.19-0.88\mathit{i} & 1.54-0.12\mathit{i} & -0.93+0.57\mathit{i} \\
0.58-0.81\mathit{i} & -0.46+0.06\mathit{i} & 0.6+0.77\mathit{i} & 0.04-0.56\mathit{i} & 0.59-0.73\mathit{i} & -0.25+0.0\mathit{i} & 0.14-0.51\mathit{i} \\
0.72-0.79\mathit{i} & 0.92-1.27\mathit{i} & 0.33+0.38\mathit{i} & 1.78-1.0\mathit{i} & 0.64+1.39\mathit{i} & -1.98+1.87\mathit{i} & 1.05-1.57\mathit{i} \\
\end{bsmallmatrix}.
\]

To obtain the map $ L \colon \PP^6 \hookrightarrow \PP^{14}$ with $ L(\Gamma) \subseteq X_8$, one only needs to compose $\widehat L$ with the previous coordinate transformations $\Gamma \leadsto \Gamma_{\text{SNF}}$:
\[
\mathbb{L}\cap X_8 = \widehat L(\Gamma_{\text{SNF}}) = \widehat L\left( (I_7+S)\cdot \Gamma_{\text{ONF}} \right) = \widehat L\left( (I_7+S)\cdot A^{-1} \cdot \Gamma \cdot \Lambda_{\text{scale}}\right).
\]
Since the configuration $\Gamma$ does not change when multiplying on the right by a diagonal matrix, we obtain the solution for the Mukai lifting problem for $\Gamma $ as:
\[
 L =  \widehat L\cdot (I_7+S)\cdot A^{-1}.
\]
We can verify $ L(\Gamma) = \mathbb{L} \cap X_8$ by evaluating $ L(\gamma_i)$ in the Plücker relations
\begin{minted}{julia}
L = L_hat*(I+S)*inv(A)
maximum([ norm(q(L*Gamma[:,i]), Inf) for i = 1:14]) #7.907992711624597e-12
\end{minted}

\section*{Outlook}

The last section illustrated the procedure developed in this paper to compute a solution for the Mukai lifting problem of a configuration of self-dual points in $\PP^6$. We observe that using our approach, the linear space will never be defined over the real numbers, even if $\Gamma$ is real. It is an interesting question how to extend this approach to find real spaces for real self-dual point configurations.

This project focused on $\mathcal{A}_6$, though our methods can also be applied to $\mathcal{A}_5$, replacing $\Gr(2,6)$ with $X_7 = \operatorname{LG}_+(5,10) \subseteq \PP^{15}$. It would be an interesting challenge to devise an algorithm for the Mukai lifting of general canonical curves $C \subseteq \PP^{7}$, \ie for general points $[C] \in \mathcal{M}_8$.

As mentioned in \Cref{Lifting using HC}, Petrakiev's map $f$ (\Cref{thm:petrakiev}) is conjecturally birational modulo $\SL(6)$. Future work on the (numerical) orbit membership problem could complete the steps outlined in \Cref{sec:testing_petrakiev}.
This may also be related to the degree of the restriction $\hat{f}_{| W}$ to a general $69$-dimensional affine subspace $W \subseteq \C^{7\times 15}$. This degree is expected to be larger than $1$ even if $f$ is birational, as mentioned in \Cref{Lifting using HC}. Such an experimental study would also benefit immensely from a faster evaluation of our polynomial system.

\section*{Acknowledgements}

We are grateful to Bernd Sturmfels for introducing us this problem, and to Simon Telen and Marta Panizzut for their suggestions regarding the implementation and their comments on the manuscript. We thank Fulvio Gesmundo, Christian Lehn, Kristian Ranestad and Bernhard Reinke for useful discussions. We thank Alessio Caminata for pointing out a mistake in \Cref{cor:SOparamDeg} in a previous version. There are no relevant financial or non-financial competing interests to report.

\printbibliography

\noindent{\bf Authors' addresses:}
\medskip

\noindent Barbara Betti,  Max Planck Institute for Mathematics in the Sciences, Leipzig, Germany \\
\hfill {\tt betti@mis.mpg.de} \medskip

\noindent Leonie Kayser,  Max Planck Institute for Mathematics in the Sciences, Leipzig, Germany \\
\hfill {\tt kayser@mis.mpg.de}

\end{document}